\documentclass[a4paper]{article}

\usepackage{amsfonts}
\usepackage{amssymb}
\usepackage{graphicx}
\usepackage{amsmath}
\usepackage[a4paper]{geometry}

\newtheorem{theorem}{Theorem}

\newtheorem{definition}{Definition}

\newtheorem{lemma}{Lemma}

\newtheorem{proposition}{Proposition}
\newtheorem{assumptions}{Assumptions}
\newtheorem{remark}{Remark}

\newenvironment{proof}[1][Proof]{\noindent\textbf{#1.} }{\ \rule{0.5em}{0.5em}}
\def\lessim{\ \lower4pt\hbox{$\buildrel{\displaystyle <}\over\sim$}\ }

\begin{document}

\title{Uniform asymptotics for kernel density estimators with variable bandwidths}
\author{{Evarist Gin\'{e}$^{a}$\thanks{$^a$Corresponding author. Email: gine@math.uconn.edu}~ and Hailin Sang$^b$} \\
\textit{$^a$\!University of Connecticut} and \textit{$^b$\!University of Cincinnati}}
\date{January 2009}
\maketitle

\begin{abstract}
\noindent It is shown that the Hall, Hu and Marron [Hall, P., Hu, T., and Marron J.S. (1995), Improved Variable
Window Kernel Estimates of Probability Densities, {\it Annals of Statistics}, 23, 1–-10] modification of
Abramson's [Abramson, I. (1982), On Bandwidth Variation in Kernel Estimates – A Square-root Law,
{\it Annals of Statistics}, 10, 1217–-1223] variable bandwidth kernel density estimator satisfies the optimal
asymptotic properties for estimating densities with four uniformly continuous derivatives, uniformly on
bounded sets where the preliminary estimator of the density is bounded away from zero.
\end{abstract}



\hfuzz=1truein

\section{Introduction and statement of the main result}
Let $f$ be a density on the real line and let $X_i$, $i\in\mathbb N$, be independent, identically distributed random variables with distribution of  density $f$.
Abramson (1982) discovered that if in the usual kernel density estimator one allows the bandwidth $h_n$ to vary with the data according to the `square root law', that is, if one takes
\begin{equation}\label{abr1}
f_n(t)=\frac{1}{n}\sum_{i=1}^n\frac{ f^{1/2}_t(X_i)}{h_n}K\left(\frac{t-X_i}{h_n} f^{1/2}_t(X_i)\right)
\end{equation}
instead of the classical estimators with the same sequence $h_n\to 0$, where $ f_t(x)=f(x)\vee (f(t)/10)$, then a {\it bias reduction} phenomenon occurs. This has been used by Hall, Hu and Marron (1995) (following Hall and Marron (1988), corr. (1992)), McKay (1993) and Novak (1999), to propose density estimators which are {\it non-negative} at all points and which estimate $f(t)$ at any given $t$ at the $L_2$-norm loss minimax rate of $n^{-4/9}$ if the density $f$ is four times differentiable with continuous and bounded derivatives.

Of course, the expression (\ref{abr1}) is not an estimator of $f$ as it depends on the unknown $f$ through $f_t$, but it becomes one if $f$ is replaced by a preliminary estimator (based on the same data, or on an independent set of data). As in the mentioned papers, expressions such as (\ref{abr1}) will be referred to here as `ideal' estimators.

It was once believed that $f_t$ in (\ref{abr1}) could be replaced by $f$, but Terrell and Scott (1992) showed that in  this case the bias reduction at a single $t$ depends heavily on the tail of $f$ and becomes negligible in the normal case (see also Hall, Hu and Marron (1995) and McKay (1993)). Taking $f_t$ instead of $f$ as Abramson did constitutes a way to deal with the tail effects on the localities $t$. Hall, Hu and Marron (1995), McKay (1993) and Novak (1999) also devised other ways of dealing with the problem. In  particular, Hall, Hu and Marron proposed the ideal estimator
\begin{equation}\label{ideal0}
\bar f_n(t)=\frac{1}{n h_n}\sum_{i=1}^{n}K\left(\frac{t-X_i}{h_n} f^{1/2}(X_i)\right) f^{1/2}(X_i)I(|t-X_{i}|<h_nB),
\end{equation}
for some $B>0$. Novak replaces $h_n$ in the indicator by $h_n/f^{1/2}(t)$ and considers powers other than 1/2 as well, and McKay replaces $ f^{1/2}_t(x)$ in (\ref{abr1}) by a smooth function $\alpha (x) =cv^{1/2}(f(x)/c^2)$ with $v(t)=t$ for all $t\ge t_0\ge 1$ with the first four derivatives of $v$ vanishing at zero.  We will focus our attention only on the simplest of these ideal estimators, which is (\ref{ideal0}), although our results should hold for the other versions as well. The ideal estimator will only be a means to study the `true' estimator, obtained from the ideal by replacement of $f$ by a preliminary estimator.

Specifically, in this article we study the uniform approximation of a density $f$ by estimators of the form
\begin{equation}\label{realest0}
\hat f(t;h_{1,n},  h_{2,n})=\frac{1}{n h_{2,n}}\sum_{i=1}^{n}K\left(\frac{t-X_i}{h_{2,n}}\hat f^{1/2}(X_i;h_{1,n})\right)\hat f^{1/2}(X_i;h_{1,n})I(|t-X_{i}|<h_{2,n}B),
\end{equation}
where
$\hat f(x;h_{1,n})$ is the classical kernel density estimator
$$
\hat f(x;h_{1,n})=\frac{1}{n h_{1,n}}\sum_{i=1}^{n}K\left(\frac{x-X_i}{h_{1,n}}\right)
$$
and $h_{i,n}$ are two sequences of bandwidths that tend to zero as $n\to\infty$.
Ideally, we would like to prove results for $\|\hat f(t;h_{1,n},  h_{2,n})-f(t)\|_\infty$, however  controlling the bias part of this error, $|E\hat f(t;h_{1,n},  h_{2,n})-f(t)|$, seems to require that $f(t)$ be bounded away from zero, so, we will consider instead the supremum of the estimation error on the `ideal' regions
\begin{equation}\label{region0}
D_r=D_r(f):=\{t: f(t)>r, |t|<1/r\},\ \ r>0,
\end{equation}
and will eventually replace $D_r$ by a region that depends on the data only and that can be made arbitrarily close to the positivity set of $f$. (We will not display the argument $f$ in $D_r(f)$ unless confusion is possible.) It is known (Hall, Hu and Marron (1995), Novak (1999)) that the bias reduction does hold for $f$ and $K$ four times differentiable and that then one has a bias of the order of $h_{2,n}^4$. This leads almost immediately in the case of the ideal estimator, and with some relatively hard work in the case of the real estimator, to an a.s. rate of convergence of $\hat f(t;h_{1,n},  h_{2,n})-f(t)$  (fixed $t$) of the order of $n^{-4/9}$ if we take $h_{2,n}\simeq n^{-1/9}$ and  $h_{1,n}\simeq n^{-2/9}$ or of a smaller order, and this is best possible ($n^{-4/9}$ is the minimax rate in the $L_2(P)$ norm for estimating $f(t)$ four times differentiable with continuity). The minimax rate for the sup norm in this case is $(n/\log n)^{-4/9}$  and we show in this article that this rate is achieved by the estimator (\ref{realest0}) uniformly in $D_r$ and in a similar data-dependent region. (See e.g. Efromovich (1999) for minimax rates.) Concretely, we prove the following theorem, in fact, as explained below, a {\it uniform }version of it.

\begin{theorem}\label{main0}
Assume the density $f$ and its first four derivatives are uniformly continuous and bounded,  that the same is true for the kernel $K$, which, moreover is non-negative, has support contained in $[-T,T]$, $T<\infty$, integrates to 1 and is symmetric about zero. Set $h_{2,n}= ((\log n)/n)^{1/9}$ and $h_{1,n}=n^{-2/9}$ (or $h_{1,n}=n^{-(2+\eta)/9}$ for some $0\le\eta<1$), $n\in\mathbb N$. Then, for all $r>0$ and constant $B\ge T/r^{1/2}$ in the definition of $\hat f(t;h_{1,n},  h_{2,n})$ in (\ref{realest0}), we have
\begin{equation}\label{main1}
\sup_{t\in D_r}\left|\hat f(t;h_{1,n},  h_{2,n})-f(t)\right|=O_{\rm a.s.}\left(\left(\frac{\log n}{n}\right)^{4/9}\right).
\end{equation}
If $\hat D_r^n$ is defined as
\begin{equation}\label{region1}
\hat D_r^n=\left\{t: \hat f(t;h_{1,n})>2r, |t|<1/r\right\},
\end{equation}
then we also have
\begin{equation}\label{main2}
\sup_{t\in \hat D_r^n}\left|\hat f(t;h_{1,n},  h_{2,n})-f(t)\right|=O_{\rm a.s.}\left(\left(\frac{\log n}{n}\right)^{4/9}\right).
\end{equation}
\end {theorem}

(Actually, $h_{2,n}$ needs only be asymptotically of the order  $((\log n)/n)^{1/9}$ in the sense that
$0\le \liminf_n \frac{h_{2,n}}{((\log n)/n)^{1/9}}\le \limsup_n\frac{h_{2,n}}{((\log n)/n)^{1/9}}<\infty$, and the same comment applies to $h_{1,n}$, but for simplicity we will work with exact values.)

We note that, by the zero-one law,  statement (\ref{main1})  is equivalent to the existence of a finite constant $C$  such that
\begin{equation}\label{main1a}
\limsup_n\left(\frac{n}{\log n}\right)^{4/9}\sup_{t\in D_r}\left|\hat f(t;h_{1,n},  h_{2,n})-f(t)\right|=C\ \ {\rm a.s.}
\end{equation}
and likewise for (\ref{main2}). And (\ref{main1a}) holds for some $C<\infty$  if and only if there is $C'<\infty$ such that
$$\lim_{k\to\infty}\Pr\left\{\sup_{n\ge k}\left(\frac{n}{\log n}\right)^{4/9}\sup_{t\in D_r}\left|\hat f(t;h_{1,n},  h_{2,n})-f(t)\right|>C'\right\}=0.$$
So, the following definition is justified (it is similar to the definition of uniform Glivenko-Cantelli classes of functions in Dudley, Gin\'e and Zinn (1991)):

\begin{definition}
For each $n\in\mathbb N$, let $Z_n(x_1,\dots,x_n;f)$ be functions of $n$ real variables $x_1,\dots, x_n$ and of the density $f$, $f\in\cal D$, where $\cal D$ is a collection of densities.
We say that the collection of random variables $Z_n(X_1,\dots, X_n,f)$, $f\in{\cal D}$, $n\in\mathbb N$, is a.s. asymptotically of the order of $a_n$ uniformly in $f\in {\cal D}$,
$$Z_n(X_1,\dots,X_n,f)=O_{\rm a.s.}(a_n)\ \ {\rm uniformly\ in}\ f\in{\cal D},$$
if there exists $C<\infty$ such that
\begin{equation}\label{defunif}
\lim_{k\to\infty}\sup_{f\in {\cal D}}{\Pr}_f\left\{\sup_{n\ge k}\frac{1}{a_n}|Z_n(X_1,\dots,X_n,f)|>C\right\}=0,
\end{equation}
and $o_{\rm a.s.}(a_n)$ uniformly in $f\in\cal D$ if the limit (\ref{defunif}) holds for every $C>0$.
\end{definition}

For $0<C<\infty$ and non-negative function $z$ such that $z(\delta)\searrow 0$ as $\delta\searrow 0$, define the class of densities
$$
{\cal D}_{C,z}:=\Bigg\{f: f\ {\rm is\ a\ density,}\ \|f^{(k)}\|_\infty\le C, \ 0\le k\le 4,~~~~~~~~~~~~~~~~~~~~~~$$
\begin{equation}\label{De}
~~~~~~~~~~~~~~~~~~~~~~~{\rm and}\ \sup_{{t\in\mathbb R}\atop {|u|\le \delta}}\left|f^{(4)}(t+u)-f^{(4)}(t)\right|\le z(\delta), \ 0<\delta\le 1\Bigg\}
\end{equation}

Here is the stronger version of Theorem \ref{main0} that we prove in this article.
\begin{theorem}\label{mainu}
Under the hypotheses of Theorem \ref{main0} we have
\begin{equation}\label{main1'}
\sup_{t\in D_r(f)}\left|\hat f(t;h_{1,n},  h_{2,n})-f(t)\right|=O_{\rm a.s.}\left(\left(\frac{\log n}{n}\right)^{4/9}\right)\ \ {\rm uniformly\ in}\ f\in{\cal D}_{C,z}
\end{equation}
and
\begin{equation}\label{main2'}
\sup_{t\in \hat D_r^n}\left|\hat f(t;h_{1,n},  h_{2,n})-f(t)\right|=O_{\rm a.s.}\left(\left(\frac{\log n}{n}\right)^{4/9}\right)\ \ {\rm uniformly\ in}\ f\in{\cal D}_{C,z}
\end{equation}
for all $0<C<\infty$ and function $z\ge 0$ such that $z(h)\searrow 0$ as $h\searrow 0$.
\end {theorem}

It is natural that, as shown by Hall, Hu and Marron (1995), the estimator (\ref{realest0}) be locally (that is, at each point $t$) asymptotically better than the classical kernel estimator that it modifies because, after all, it is  obtained from the classical one by {\it local or spatial adaptation} of the bandwidth. This theorem shows that, up to a logarithmic factor, the improvement is not only local but holds uniformly over all $t$ for which $f(t)$ is slightly above zero, and uniformly as well over large classes of densities with four continuous derivatives. This may seem surprising and is certainly desirable. See the comments by Donoho, Johnstone, Kerkyacharian and Picard (1995) about the scarcity of theoretical results on `spatially adaptive' estimators.

We do not know of any other non-negative estimators of a density that achieve such good rates in sup-norm loss (although Abramson's or Novak's may). Thresholding  wavelet density estimators (Donoho, Johnstone, Kerkyacharian and Picard (1996)) constitutes also a kind of adaptation to the local behavior of $f$ since wavelets pick up local behavior; these estimators may not be non-negative on the whole domain, but are rate adaptive to the smoothness of $f$ in sup-norm loss, in particular satisfying Theorem  \ref{mainu} -but also attaining the rate $((\log n)/n)^{t/(2t+1)}$ uniformly on densities in the unit ball of $C^t({\mathbb R})$ (Gin\'e and Nickl (2008)). See also Gin\'e and Nickl (2009) for estimators with this property based on convolution kernels of higher order and Lepski's method.

We  first prove Theorem \ref{mainu} for the ideal estimator and then  show that the supremum over $D_r$ of the difference between the true and the ideal estimators is of the order of $\left((\log n)/n\right)^{4/9}$. For this we  use empirical process and U-process techniques: basically, the classes of functions involved in the supremum in (\ref{main1}) and in other suprema appearing in the proofs are of Vapnik-\v Cervonenkis type (see e.g. de la Pe\~na and Gin\'e (1999)) and therefore we can use the appropriate version of Talagrand's exponential inequality for empirical processes (as in Einmahl and Mason (2000) and Gin\'e and Guillou (2002)), and an inequality due to Major (2006) for $U$-processes. We relegate to an appendix proving that the relevant classes of functions are of VC type, so that we get this technicality out of the way in the main proofs.

Since we use empirical processes, in order to avoid measurability problems and  without loss of generality, we assume throughout that the variables $X_i$ are the coordinate functions on $\Omega={\mathbb R}^{\mathbb N}$, equipped with the product $\sigma$-algebra and the probability measure  $\Pr=P^{\mathbb N}$, $dP(x)=f(x)dx$, that we will denote as ${\Pr}_f$ if (and only if) we need to distinguish among several densities.

\section{The ideal estimator}

In this section we obtain the asymptotic size of the uniform deviation of the ideal estimator (\ref{ideal0}) from the density $f$, that is, we will consider the a.s. asymptotic size of
$$\sup_{t\in D_r} |\bar f(t;h_n)-f(t)|:=\|\bar f(t;h_n)-f(t)\|_{D_r}$$
As usual this quantity is divided into the bias part, $\|E\bar  f(t;h_n)-f(t)\|_{D_r}$, and the stochastic part or variance part $\|\bar f(t;h_n)-E\bar f(t;h_n)\|_{D_r}$. Each is studied in a different subsection. There is no problem with extending the supremum for the variance part over the whole of $\mathbb R$; the problem is, as mentioned above, with the bias.

We  will use the shorthand notations
$$\bar f_n(t;h)=\bar f_n(t)=\bar f(t;h_n)$$
so that we display only either $h_n$ or $n$ but not both; the first expression is used in this section and the second in the next.

\subsection{Stochastic part of the `ideal' estimator}

In this subsection we assume:
\begin{assumptions} \label{ass1}
The sequence $h_n$ will satisfy the following classical conditions:
\begin{equation}\label{band}
h_{n}\searrow0,  \;\; \frac{nh_{n}}{|\log h_{n}|}\rightarrow\infty, \;\;  \frac{|\log h_{n}|}{\log\log{n}}
\rightarrow\infty,\;\; and\:\: nh_n\nearrow,
\end{equation}
as $n\to\infty$. The kernel $K$ will be a non-negative left or right continuous function, bounded, with support contained in $[-T,T]$ for some $T<\infty$,  and of bounded variation. $f$ is a bounded density.
\end{assumptions}

The proof of the following proposition is patterned after the proof of a similar theorem in Gin\'e and Guillou (2002), and it consists of blocking and application of  Talagrand's inequality (\ref{tal}). It extends to the variable bandwidth estimator a well known uniform rate for the usual kernel estimator (Silverman (1978), formula (9)).

\begin{proposition}\label{varid}
Under the hypotheses in Assumptions \ref{ass1},
\[
||\bar{f}_n-E\bar{f}_n||_{\infty}=O_{\rm a.s.}\left(\sqrt{\frac{\log h_{n}^{-1}}{n h_{n}}}\right)
\]
uniformly over all densities $f$ such that $\|f\|_\infty\le C$, for any $0<C<\infty$.
\end{proposition}
%proof of Theorem 1
\begin{proof} We block the terms between dyadic integers as follows, where, for ease of notation we set
$\textbf{1}_{i,n}(t):=I(|t-X_i|<h_nB)$ and $\textbf{1}_{ih}(t)=I(|t-X_{i}|<hB)$:
\begin{eqnarray}\label{eq1}
&&\Pr\left\{\max_{2^{k-1}<n\le 2^{k}}\sqrt{\frac{n h_{n}}{\log h_{n}^{-1}}}
||\bar{f}_n-E\bar{f}_n||_{\infty}>\lambda\right\} \notag\\
&&~~~\le\Pr\Biggr\{\max_{2^{k-1}<n\le 2^{k}}\sqrt{\frac{1}{2^{k-1} h_{2^k}\log h_{2^k}^{-1}}}\sup_{t\in \mathbb{R}}\left|\sum_{i=1}^{n}\left[
K\left(\frac{t-X_{i}}{h_{n}}f^{1/2}(X_{i})\right)f^{1/2}(X_{i})\textbf{1}_{i,n}(t)
\right.\right.{}\nonumber\\
&&~~~~~~~~~~~~~~~~~~~~~~~\left.\left.{} -EK \left(\frac{t-X_{i}}{h_{n}}f^{1/2}(X_{i})\right)f^{1/2}(X_{i})\textbf{1}_{i,n}(t)\right]\right|
>\lambda\Biggr\} \nonumber \\
&&~~~\le\Pr\Biggr\{\max_{2^{k-1}<n\le 2^{k}}
\sup_{ {t\in \mathbb{R}}\atop {h_{2^{k}}\le h<h_{2^{k-1}}} }\left|\sum_{i=1}^{n}\left[
K\left(\frac{t-X_{i}}{h}f^{1/2}(X_{i})\right)f^{1/2}(X_{i})\textbf{1}_{ih}(t)-
\right.\right.\nonumber\\
&&~~~~~~~~~~~~~~~~~~~~~~~\left.\left.-EK\left(\frac{t-X_{i}}{h}f^{1/2}(X_{i})\right)f^{1/2}(X_{i})\textbf{1}_{ih}(t)\right]\right|
>\lambda\sqrt{2^{k-1} h_{2^k}\log h_{2^k}^{-1}}\Biggr\}
\end{eqnarray}
for any $\lambda>0$, where we used that $h_n$ decreases and that the function $x\log x^{-1}$ is decreasing for $x\le 1/e$. As we see in the Appendix the class of functions
\begin{equation}
\mathcal{F}=\left\{K\left(\frac{t- \cdot}{h}f^{1/2}(\cdot)\right)f^{1/2}(\cdot)I(
|t-\cdot|<hB):t\in\mathbb{R}, h>0 \right\} \label{entr0}
\end{equation}
is a bounded VC class of measurable functions with respect to the constant envelope
$W:=\|K\|_V||f||_{\infty}^{1/2}$,
where $||K||_V$ is the total variation norm of $K$.
Hence, the subclasses
\begin{equation}
\mathcal{F}_{k}=\left\{K\left(\frac{t-\cdot}{h}f^{1/2}(\cdot)\right)f^{1/2}(\cdot)I(
|t-\cdot|<hB):t\in\mathbb{R}, h_{2^{k}}\le h<h_{2^{k-1}} \right\} \label{entr3}
\end{equation}
are VC classes of functions with respect to $U_{k}=W$ also and with the same characteristics $A(v)$ and $v$ as $\cal F$. Next, in order to apply
 Talagrand's inequality (\ref{tal}), we obtain a sensible bound $\sigma^2_k$ for the maximum variance of the functions in ${\cal F}_k$:
 \begin{eqnarray}
&&\frac{1}{h}\int_{\mathbb{R}}K^{2}\left(\frac{t-x}{h}f^{1/2}(x)\right)  I(
|t-x|<hB)f^2(x)dx
\le\frac{1}{h}\int_{\mathbb{R}}K^{2}\left(\frac{t-x}{h}f^{1/2}(x)\right)
f^{2}(x)dx\nonumber\\
&&~~~~~~~~~~~~~~~~~~~~~~~~~~~~~~~~~=\int_{\mathbb{R}}K^{2}\left(uf^{1/2}(t-hu)\right)f^{2}(t-hu)du \nonumber\\
&&~~~~~~~~~~~~~~~~~~~~~~~~~~~~~~~~~\le\int_{\mathbb{R}}(||K||_{\infty}^{2}||f||_{\infty}^{2})\wedge (||K||_{\infty}^{2}
(T/|u|)^4)du \notag\\
&&~~~~~~~~~~~~~~~~~~~~~~~~~~~~~~~~~=2||K||_{\infty}^{2}\left[||f||_{\infty}^{2}\int_{0}^{T/||f||_{\infty}^{1/2}}du
+\int_{T/||f||_{\infty}^{1/2}}^{\infty}
\left(\frac{T}{u}\right)^4du\right]\nonumber\\
&&~~~~~~~~~~~~~~~~~~~~~~~~~~~~~~~~~=\frac{8}{3}T||K||_{\infty}^{2}||f||_{\infty}^{3/2}.
\end{eqnarray}
So, we can take $\sigma_k^2:=\frac{8}{3}T||K||_{\infty}^{2}||f||_{\infty}^{3/2}2h_{2^{k}}$ (using the fourth condition in (\ref{band})).
$U_k=W$ is eventually much larger than $\sigma_k$ and
$$\sqrt{2^k}\sigma_k\sqrt{\log \frac{U_k}{\sigma_k}}<<2^k\sigma_k^2$$
by the second condition in (\ref{band}) (here and elsewhere, the sign $<<$ should be read as `of smaller order than' when the indexing variable, in this case $k$, tends to infinity). If $\lambda$ in (\ref{eq1}) is taken to be  large enough so that
\begin{equation}\label{tcond}
C_1\sqrt{2^k}\sigma_k\sqrt{\log \frac{RU_k}{\sigma_k}}<\lambda\sqrt{2^{k-1} h_{2^k}\log h_{2^k}^{-1}}<<2^k\sigma_k^2,
\end{equation}
where $C_1$ is one of the constants in Talagrand's inequality (\ref{tal}), then, this inequality applied to the inequalities (\ref{eq1}),  gives
\begin{equation}\label{tineq}
\Pr\left\{\max_{2^{k-1}<n\le 2^{k}}\sqrt{\frac{n h_{n}}{\log h_{n}^{-1}}}
||\bar{f}_n-E\bar{f}_n||_{\infty}>\lambda\right\}\le C_2\exp\left(-\frac{C_3\lambda^2 2^{k-1} h_{2^k}\log h_{2^k}^{-1}}{2^k\frac{16}{3}T||K||_{\infty}^{2}||f||_{\infty}^{3/2}h_{2^k}}\right).
\end{equation}
Set
$\lambda=L\sqrt{T}\|K\|_\infty C^{3/4}$. Then we can choose  $L$ large enough such that inequality (\ref{tcond}) is satisfied for all $k>k_0$, $k_0$ depending on $K$ only, and for this $\lambda$ inequality (\ref{tineq}) becomes
$$\sup_{f:\|f\|_\infty\le C} {\Pr}\left\{\max_{2^{k-1}<n\le 2^{k}}\sqrt{\frac{n h_{n}}{\log h_{n}^{-1}}}
||\bar{f}_n-E\bar{f}_n||_{\infty}>\lambda\right\}\le C_2\exp\left(-\frac{3C_3L^2 h_{2^k}\log h_{2^k}^{-1}}{2^5}\right),
$$
where
the term at the right hand side  is the general term of a convergent series because $(\log h_{2^k}^{-1})/\log k\to\infty$ by the third inequality in (\ref{band}). This proves the proposition.
\end{proof}

\medskip
This result, which is good enough for our purposes,  can possibly be made more precise for each particular density $f$: for instance, Sang (2008) proves
\[
\lim_{n\rightarrow \infty}\sqrt{\frac{n h_{n}}{\log h_{n}^{-1}}}
||\bar{f}_n-E\bar{f}_n||_{\infty}=\|K\|_2\|f\|_\infty^{3/4}\;\;a.s.
\]
if the ideal Hall, Hu, Marron estimator is replaced by the ideal Novak estimator with $\alpha=1/2$, and under some additional, natural assumptions. This suggests that the rate in Proposition 1 is optimal. Also, Theorem \ref{varid} admits more general and stronger versions: see Mason and Swanepoel (2008) for a recent result  along the lines of the previous theorem, with uniformity in bandwidth added, and for a general class of estimators that includes ours.

\subsection{Bias of the `ideal' estimator}
The assumptions on $f$, $K$ and $h_n$ in this section are as follows:

\begin{assumptions}\label{ass2}
We assume that the densities $f$ and the kernel $K$ as well as their  first four derivatives are bounded and uniformly continuous, and moreover that $K$ has support contained in $[-T,T]$, $T<\infty$, it integrates to 1 and is symmetric about zero.  We also assume $h_n\to0$ as $n\to\infty$ (and $h_n>0$).
\end{assumptions}

We set
\begin{eqnarray}\label{ftilde}
\tilde f(t;h)&:=&E\bar f_n(t;h)=\frac{1}{ h}\int f^{3/2}(x)K\left({x-t\over h}f^{1/2}(x)\right)I(|x-t|< Bh)dx\notag\\
&=&\int_{-B}^{B}
f^{3/2}(t+hw)K(wf^{1/2}(t+hw))dw=\int_{-B}^Bg_{t,w}(hw)dw,\end{eqnarray}
where, for $t$ and $w$  fixed,
\begin{equation}\label{g}
g_{t,w}(u)=f^{3/2}(t+u)K(wf^{1/2}(t+u)).
\end{equation}
If no confusion may arise, we drop the subindices $t,w$ from $g$.
To estimate the bias of the ideal estimator, $\tilde f(t;h)-f(t)$, one develops $g(hw)$ about
zero and integrates. For further reference, we record the first four derivatives of $g(u)$:
by direct computation or e.g. from Novak (1999), we  have,  with $r(u)=f^{3/2}(t+u)$ and $s(u)=wf^{1/2}(t+u)$,
$$g(u)=r(u)K(s(u)),\ \ g'(u)=r'(u)K(s(u))+r(u)s'(u)K'(s(u)),$$
and, dropping the arguments for simplicity,
\begin{eqnarray}\label{deriv}
g''&=&r''K +(2r's'+rs'')K'+r(s')^2K'',\notag\\
g'''&=&r'''K+(3r''s'+3r's''+rs''')K'+3(r'(s')^2+rs's'')K''+r(s')^3K'''\notag\\
g^{(4)}&=&r^{(4)}K+(4r'''s'+6r''s''+4r's'''+rs^{(4)})K'+(6r''(s')^2+12r's's''+4rs's'''
\notag\\
&&~~~~~~~~~~~~~~+3r(s'')^2)K''+
(4r'(s')^3+6r(s')^2s'')K'''+r(s')^4K^{(4)}.
\end{eqnarray}

\begin{proposition}\label{biasid}
Under the hypotheses in Assumptions \ref{ass2},  if the constant $B$ in the definition of $\tilde f_n(t;h_n)$ satisfies $B\ge T/r^{1/2}$, then, for all $0<C<\infty$ and functions $z\ge0$ with $z(h)\searrow 0$ as $h\searrow 0$, we have
\begin{equation}\label{unibias}
\lim_{n\to\infty}\sup_{f\in{\cal D}_{C,z}}\sup_{t\in D_r}\left|\frac{\tilde f(t;h_n)-f(t)}{h_n^4}-H(t,f,K)\right|=0
\end{equation}
and
\begin{equation}\label{H}
\sup_{f\in{\cal D}_{C,z}}\sup_{t\in D_r}|H(t,f,K)|<\infty,
\end{equation}
where
$$H(t,f,K)=\left[\frac{(f')^4(t)}{ f^5(t)}-\frac{3(f')^2(t)f''(t)}{ 2f^4(t)}+\frac{4f'(t)f'''(t)+3(f'')^2(t)}{ 12f^3(t)}-\frac{f^{(4)}(t)}{24 f^2(t)}\right]\int  v^4K(v)dv.$$
\end{proposition}

\begin{proof} Since $f$ and $K$ and their first four derivatives are continuous and $f>r/2$ on a neighborhood of $D_r$, it follows that, if $g_{tw}$ is as defined in (\ref{g}), there exists $n_0<\infty$ such that, for all $t\in D_r$ and for all $w\in{\mathbb R}$,  $g_{t,w}^{(4)}$ is continuous on $[-Bh_n,Bh_n]$ for all $n\ge n_0$: note that $g^{(4)}(u)$ is  a linear combination of  $K$ and its first four derivatives at $wf^{1/2}(t+u)$ whose coefficients are fractions that have products of powers of $w$ and powers of $f(t+u)$ and its derivatives in the numerator, and powers of $f(t+u)$ in the denominator (see(\ref{deriv})). Therefore, Taylor expansion gives
\begin{equation}\label{taylor}
g(u)=\sum_{k=0}^3g^{(i)}(0)\frac{u^i}{ i! }+\frac{u^4}{ 4!}E_\tau g^{(4)}(\tau u)
\end{equation}
where $\tau$ is a random variable with density $\lambda(x)=4(1-x)^3$, $0\le x\le 1,$ that does not depend on $t$, $w$ or $u$, and $E_\tau$ denotes expectation with respect to this variable. Equation (\ref{taylor}) can be easily verified by integration by parts in $\int_0^14(1-t)^3g^{(4)}(tu)dt$. Next note that
\begin{equation}\label{firstt}
\int_{-B}^Bg_{t,w}(0)dw=\int_{-B}^Bf^{3/2}(t)K(wf^{1/2}(t))dw=\int_{-Bf^{1/2}(t)}^{Bf^{1/2}(t)}f(t)K(v)dv=f(t)
\end{equation}
since the support of $K$ is contained in $[-Bf^{1/2}(t),Bf^{1/2}(t)]$ by the hypothesis on $B$ and since $K$ integrates to 1. Further, since $s'$ contains a $w$ factor, there are functions $c_i(f,t)$, $i=1,2$, such that
$$\int_{-B}^Bwg_{t,w}'(0)dw=\int_{-Bf^{1/2}(t)}^{Bf^{1/2}(t)}(c_1(f,t)vK(v)+c_2(f,t)v^2K'(v))dv=0$$
because $K$ is even and $K'$ is odd. Similarly (that is, using only the symmetry properties of $K$ and its derivatives), we also get $\int_{-B}^B w^3g_{t,w}'''(0)dw=0.$ That these two integrals vanish is obvious and not surprising; what is remarkable is that also
$\int_{-B}^Bw^2g_{t,w}''(0)dw=0$, and this fact is the main reason for the bias reduction achieved by Abramson's (1982) `inverse square root rule'. We sketch an argument for completeness. Note first that, from the expression for $g''$ in (\ref{deriv}), integrating by parts,
\begin{eqnarray*}
\int_{-B}^Bw^2[r(s')^2K''(s)](0)dw&=&{1\over 4}f^{1/2}(t)(f')^2(t)\int_{-B}^Bw^4K''(wf^{1/2}(t))dw\\
&=&-(f')^2(t)\int_{-B}^Bw^3K'(wf^{1/2}(t))dw.
\end{eqnarray*}
Collecting terms in $K'$, this gives
\begin{eqnarray*}\int_{-B}^Bw^2g_{t,w}''(0)dw&=&\left[{3\over 4}f^{-1/2}(t)(f')^2(t)+{3\over 2}f^{1/2}(t)f''(t)\right]\int_{-B}^Bw^2K(wf^{1/2}(t))dw\\
&&~~~~~~~~+\left[{1\over 4}(f')^2(t)+{1\over 2}f(t)f''(t)\right]\int_{-B}^Bw^3K'(wf^{1/2}(t))dw,
\end{eqnarray*}
and, integrating by parts the second integral, we get zero. [See Novak (1999) for a proof that, if one replaces $f^{1/2}$ by $f^\alpha$ (and $f^{3/2}$ by $f^{\alpha +1}$) in the definition of $\hat f_n(t;h)$ the only $\alpha$ for which $\int_{-B}^Bw^2g''(0)dw=0$ for all $f$ twice differentiable with $f(t)\ne0$ is $\alpha=1/2$.]. Thus, we have
$$\int_{-B}^Bw^2g^{(i)}_{t,w}(0)dw=0\ \ {\rm for}\ \ i=1,2,3,$$
and we conclude, from this, (\ref{ftilde}), (\ref{taylor}) and (\ref{firstt}), that
\begin{equation}\label{ftilde2}
\tilde f(t;h)=\int_{-B}^Bg_{t,w}(hw)dw=f(t)+\frac{h^4}{4!}\int_{-B}^Bw^4E_\tau g^{(4)}(\tau h w)dw.
\end{equation}
Using the formula  for $g^{(4)}$ in (\ref{deriv}), integrating by parts and collecting terms, it is tedious but straightforward to check that
\begin{eqnarray}\label{rem}
&&\int_{-B}^Bw^4g_{t,w}^{(4)}(0)dw\\
&&~~~~~~~~=\left[\frac{24(f')^4(t)}{f^5(t)}-\frac{36(f')^2(t)f''(t)}{ f^4(t)}+\frac{8f'(t)f'''(t)+6(f'')^2(t)}{f^3(t)}-\frac{f^{(4)}(t)}{ f^2(t)}\right]\int  v^4K(v)dv,\notag
\end{eqnarray}
and to note that
\begin{equation}\label{rem2}
\sup_{f\in {\cal D}_{C,z}}\sup_{t\in D_r}\left|\int_{-B}^Bw^4g_{t,w}^{(4)}(0)dw\right|<\infty.
\end{equation}

Now, the boundedness and uniform continuity of  $K$ and its four derivatives and the facts that, for $f\in{\cal D}_{C,z}$, $f$ and its  first three derivatives are Lipschitz with common constant $C$ and the fourth derivatives $f^{(4)}$ have all the same modulus of continuity $z$ at all $t$, and  that $f$ is bounded away from zero in a neighborhood of $D_r$, imply that
\begin{equation}\label{comp}
\lim_{n\to\infty}\sup_{f\in {\cal D}_{C,z}}\sup_{0\le \tau\le 1}\sup_{w\in[-B,B]}\sup_{t\in D_r}|g_{t,w}^{(4)}(\tau h_nw)-g_{t,w}^{(4)}(0)|=0.
\end{equation}
Therefore,
$$\lim_{n\to\infty}\sup_{f\in {\cal D}_{C,z}}\sup_{t\in D_r}\left|\int_{-B}^Bw^4E_\tau (g_{t,w}^{(4)}(\tau h_nw)-g_{t,w}^{(4)}(0))dw\right|=0$$
and we have from this and  (\ref{ftilde2})  that
\begin{eqnarray*}
&&\sup_{f\in {\cal D}_{C,z}}\sup_{t\in D_r} \left|h_n^{-4}(\tilde f(t;h_n)-f(t)) -{1\over 4!}\int_{-B}^Bw^4g_{t,w}^{(4)}(0)dw\right|\\
&&~~~~~~~~~~~=\sup_{f\in {\cal D}_{C,z}}\sup_{t\in D_r}\left| {1\over 4!}E_\tau\int_{-B}^Bw^4 (g_{t,w}^{(4)}(\tau h_nw)-g_{t,w}^{(4)}(0))dw\right|\to
0\end{eqnarray*}
as $n\to\infty$. This, together with (\ref{rem}) and (\ref{rem2}) prove the proposition.
\end{proof}

\vskip.1truein
This proposition is similar to Theorem 3.1 of Hall, Hu and Marron (1995) and to Theorem 1 of Novak (1999), who do not consider uniformity in $t$ or $f$, and our proof is somewhat adapted from the latter reference (which deals with a slightly different estimator). See also Hall (1990), Terrell and Scott (1992) and McKay (1993).

\medskip
 Combining Propositions \ref{varid} and \ref{biasid} we obtain the following result for the `ideal' estimator.

\begin{theorem}\label{unifidealthm}
Under the Assumptions \ref{ass2} and with $h_n= ((\log n)/n)^{1/9}$, we have, for every $0<C<\infty$ and function $Z$ such that $z(h)\searrow 0$ as $h\searrow 0$, for all $r>0$ and constant $B\ge T/r^{1/2}$ in the definition of $\bar f_n(t;h)$ in (\ref{ideal0}),
\begin{equation}\label{main4}
\sup_{t\in D_r}\left|\bar f(t;h_n)-f(t)\right|=O_{\rm a.s.}\left(\left(\frac{\log n}{n}\right)^{4/9}\right)\ \ {\rm uniformly \ in}\ f\in{\cal D}_{C,z}.
\end{equation}
\end{theorem}

\begin{remark}{\rm The limit (\ref{comp}) is straightforward, but lengthy to compute. By way of illustration we indicate how to prove a `small piece' of it.
Let us consider, for example, the term in $f^{(4)}$ from the first summand $r^{(4)}K$ in the expression for $g^{(4)}$ in (\ref{deriv}). It is $(3/2)f^{1/2}(t+u)f^{(4)}(t+u)K(wf^{1/2}(t+u))$. Then,
$$|f^{1/2}(t+u)f^{(4)}(t+u)K(s(t+u))-f^{3/2}(t)f^{(4)}(t)K(s(t))|
\le \|f\|_\infty^{1/2}\|K\|_\infty|f^{(4)}(t+u)-f^{(4)}(t)|$$
$$+\|f^{(4)}\|_\infty\|K\|_\infty|f^{1/2}(t+u)-f^{1/2}(t)|+\|f^{(4)}\|_\infty \|f\|_\infty^{1/2}|K(s(t+u))-K(s(t))|.$$
And we have, for the first summand,
$$
|f^{(4)}(t+\tau h_nw)-f^{(4)}(t)|\le z(Bh_n)\to 0$$
uniformly in $t$ and $f$ (recall $|\tau|\le 1$, $|w|\le B$). For the second summand, for $n$ large enough,
$$|f^{1/2}(t+u\tau h_nw)-f^{1/2}(t)|\le\frac{|f(t+u\tau h_nw)-f(t)|}{r^{1/2}}\le\frac{CBh_n}{r^{1/2}}\to0,$$
and the limit zero for the third follows directly by uniform continuity of $K$ and the common Lipschitz constant $C$ for all $f\in{\cal D}_{C,z}$.}
\end{remark}

\section{Comparison between the ideal and the true estimators}

In this section we make the following assumptions on the kernel $K$, the densities $f$ and the band sequences:

\medskip
\begin{assumptions}\label{ass3}
 We assume that $K$ is supported by $[-T,T]$ for some $T<\infty$ and that it has a uniformly bounded second derivative.  We also assume that the densities $f$ are bounded and have at least two bounded derivatives,
\begin{equation}{pc}
f\in{\cal P}_C:=\{f\ {\rm is\ a\ density}:\|f^{(k)}\|_\infty\le C, 0\le k\le 2\}
\end{equation}
for some $C<\infty$. We set $h_{1,n}=n^{-2/9}$ and $h_{2,n}=((\log n)/n)^{1/9}$, $n\in\mathbb N$.
 \end{assumptions}

\medskip
Let
\begin{equation}\label{realest}
\hat f(t;h_{1,n},  h_{2,n})=\frac{1}{n h_{2,n}}\sum_{i=1}^{n}K\left(\frac{t-X_i}{h_{2,n}}\hat f^{1/2}(X_i;h_{1,n})\right)\hat f^{1/2}(X_i;h_{1,n})I(|t-X_{i}|<h_{2,n}B),
\end{equation}
where
$\hat f(x;h_{1,n})$ is the classical kernel density estimator
\begin{equation}\label{real}
\hat f(x;h_{1,n})=\frac{1}{n h_{1,n}}\sum_{i=1}^{n}K\left(\frac{x-X_i}{h_{1,n}}\right).
\end{equation}

The object of this section consists in proving that
\begin{equation}\label{diff}
\hat f(t;h_{1,n},  h_{2,n})-\bar f(t;h_{2,n})
\end{equation}
 is asymptotically almost surely of the order of $\sqrt{(\log h_{2,n}^{-1})/(nh_{2,n})}$ uniformly in $t$ on the region $D_r$ defined in (\ref{region0}), for any $r>0$, if we take
$h_{2,n}=\left((\log n)/n\right)^{1/9}$ and $h_{1,n}=n^{-2/9}$. Note that $h_{2,n}$ is the optimal rate  `up to a log' given the order of the bias, whereas the preliminary estimator has a bandwidth sensibly smaller than the optimal $n^{-1/5}$ (it is less smooth than the optimal, `undersmoothed') and therefore its bias will be negligible with respect to its variance term. The main result of this paper will follow from this analysis and the result from the `ideal' estimator.

We follow the pattern in Hall and Marron (1988) and Hall, Hu and Marron (1995)  for the linearization of (\ref{diff}), with significant differences in order to account for the uniformity in $t$. For instance, they do not necessarily undersmooth the preliminary estimator (whereas we believe one should) and, moreover, we are required to use empirical and U-process theory. We adhere to their notation as much as possible.

The first step is to notice that, if we define $\delta_n(t)$ by the equation
\begin{equation}\label{delta}
\delta_n(t)=\frac{\hat f^{1/2}(t;h_{1,n})-f^{1/2}(t)}{f^{1/2}(t)}=\frac{\hat f(t;h_{1,n})-f(t)}{(\hat f^{1/2}(t;h_{1,n})+f^{1/2}(t))f^{1/2}(t)},
\end{equation}
so that $\hat f^{1/2}(t;h_{1,n})=f^{1/2}(t)(1+\delta_n(t))$, then
we have
\begin{equation}\label{zero}
\sup_{t\in D_r^\varepsilon}\delta_n(t)=o_{\rm a.s.}(1)\ \ {\rm uniformly\ in}\ f\ {\rm such\ that}\ \|f\|_\infty\le C,
\end{equation}
where $D_r^\varepsilon$ denotes the $\varepsilon$-neighborhood of $D_r$ for $\varepsilon$ such that $f(t)>r/2$ in $D_r^\varepsilon$ ($f$ is uniformly continuous). We drop the subindex $n$ from $\delta$ from now on. Set
$$D(t;h_{1,n})=\hat f(t;h_{1,n})-E\hat f(t;h_{1,n})\ \ {\rm and}\ \ b(t;h_{1,n})=E\hat f(t;h_{1,n})-f(t)$$
and note that
\begin{equation}\label{classic1}
\|D(\cdot;h_{1,n})\|_\infty=O_{a.s.}\left(\sqrt{\frac{\log h_{1,n}^{-1}}{nh_{1,n}}}\right)\ \ {\rm uniformly\ in}\ f\ {\rm such\ that}\ \|f\|_\infty\le C
\end{equation}
for all $0<C<\infty$ by a result in Deheuvels (2000) and in Gin\'e and Guillou (2002), and that
\begin{equation}\label{classic2}
 \|b(\cdot;h_{1,n})\|_\infty\le \left(\int K(u)u^2du\right)\|f''\|_\infty h_{1,n}^2
\end{equation}
by the classical bias computation for symmetric kernels. Since the numerator in the expression at the right hand side (\ref{delta}) is just $D(t)+b(t)$ and the denominator is not smaller than $f(t)$ which is in turn larger than $r/2$, (\ref {zero}) follows from (\ref{classic1}) and (\ref{classic2}). Define
$$L_1(z)=zK'(z)\ \ {\rm and}\ \  L(z)=K(z)+zK'(z),\ \ z\in\mathbb R.$$
We then have
\begin{eqnarray*}
K\left(\frac{t-X_i}{h_{2,n}}\hat f^{1/2}(X_i;h_{1,n})\right)
&=&K\left(\frac{t-X_i}{h_{2,n}}f^{1/2}(X_i)+
\frac{t-X_i}{h_{2,n}}f^{1/2}(X_i)\delta (X_i)\right)\\
&=&K\left(\frac{t-X_i}{h_{2,n}}f^{1/2}(X_i)\right)\\
&&~~~+K^\prime\left(\frac{t-X_i}{h_{2,n}}f^{1/2}(X_i)\right)
\frac{t-X_i}{h_{2,n}}f^{1/2}(X_i)\delta (X_i)+\delta_2(t,X_i)\\
&=&K\left(\frac{t-X_i}{h_{2,n}}f^{1/2}(X_i)\right)+L_1\left(\frac{t-X_i}{h_{2,n}}f^{1/2}(X_i)\right)
\delta (X_i)+\delta_2(t,X_i),
\end{eqnarray*}
where
\begin{equation}\label{delta2}
\delta_2(t,X_i)=\frac{K^{\prime\prime}(\xi)}{2}
\frac{(t-X_i)^2}{h_{2,n}^2}f(X_i)\delta^2(X_i),
\end{equation}
$\xi$  being a (random) number between $\frac{t-X_i}{h_{2,n}}f^{1/2}(X_i)$ and $\frac{t-X_i}{h_{2,n}}f^{1/2}(X_i)+\frac{t-X_i}{h_{2,n}}f^{1/2}(X_i)\delta (X_i).$
Then, plugging this development and that of $\hat f^{1/2}$ in the definition (\ref{realest}) of $\hat f$, we obtain
\begin{eqnarray}
\hat f(t;h_{1,n},  h_{2,n})\!\!\!&=&\!\!\!\bar f(t;h_{2,n})\nonumber\\
&&\!\!\!\!+\frac{1}{n h_{2,n}}\sum_{i=1}^{n}L_1\left(\frac{t-X_i}{h_{2,n}}f^{1/2}(X_i)\right)
f^{1/2}(X_i)\delta (X_i)I(|t-X_{i}|<h_{2,n}B)\nonumber\\
&&\!\!\!\!+\frac{1}{n h_{2,n}}\sum_{i=1}^{n}K\left(\frac{t-X_i}{h_{2,n}}f^{1/2}(X_i)\right)
f^{1/2}(X_i)\delta (X_i)I(|t-X_{i}|<h_{2,n}B)\nonumber\\
&&\!\!\!\!+\frac{1}{n h_{2,n}}\sum_{i=1}^{n}f^{1/2}(X_i)\delta_2 (t,X_i))I(|t-X_{i}|<h_{2,n}B)\label{e1}\\
&&\!\!\!\!+\frac{1}{n h_{2,n}}\sum_{i=1}^{n}L_1\left(\frac{t-X_i}{h_{2,n}}f^{1/2}(X_i)\right)
f^{1/2}(X_i)\delta^2 (X_i)I(|t-X_{i}|<h_{2,n}B)\label{e2}\\
&&\!\!\!\!+\frac{1}{n h_{2,n}}\sum_{i=1}^{n}
f^{1/2}(X_i)\delta (X_i)\delta_2 (t,X_i))I(|t-X_{i}|<h_{2,n}B)\label{e3}\\
\!\!\!&=&\!\!\!\bar{f}(t;h_{2,n})+\delta_3(t)\nonumber\\
&&\!\!\!\!+\frac{1}{n h_{2,n}}\sum_{i=1}^{n}L\left(\frac{t-X_i}{h_{2,n}}f^{1/2}(X_i)\right)
f^{1/2}(X_i)\delta (X_i)I(|t-X_{i}|<h_{2,n}B),\label{e4}
\end{eqnarray}
where $\delta_3(t)$ is the sum of the terms (\ref{e1}), (\ref{e2}) and (\ref{e3}), which are of a smaller order than the term (\ref{e4}) by (\ref{zero}) and (\ref{delta2}) for $t\in D_r$ (as we will readily check). Since by (\ref{classic1}) and (\ref{classic2}), $D(y;h_{1,n})$ dominates $b(y;h_{1,n})$ uniformly in $\mathbb R$, we should further decompose (\ref{e4}) to display its $D$-part and its $b$-part. By the definitions of $\delta$, $D$ and $b$, we have
\begin{eqnarray*}
\delta(t)&=&\frac{D(t;h_{1,n})}{2f(t)}+\frac{b(t;h_{1,n})}{2f(t)}+\frac{D(t;h_{1,n})+b(t;h_{1,n})}{2f(t)}
\frac{f^{1/2}(t)-\hat f^{1/2}(t;h_{1,n})}{\hat f^{1/2}(t;h_{1,n})+f^{1/2}(t)}\\
&:=&\frac{D(t;h_{1,n})}{2f(t)}+\frac{b(t;h_{1,n})}{2f(t)}+\delta_4(t),
\end{eqnarray*}
(where $\delta_4$ depends on $n$, but we do not display this dependence) and note that (again using $(a^{1/2}-b^{1/2}=(a-b)/(a^{1/2}+b^{1/2})$),
\begin{equation}\label{delta4}
\sup_{t\in D_r^\varepsilon}|\delta_4(t)|\le \frac{1}{3r^{3/2}}\sup_{t\in D_r^\varepsilon}\left[D(t;h_{1,n})+b(t;h_{1,n})\right]^2
\end{equation}
which is small by (\ref{zero}) (note that $\delta_n$ is $D+b$ divided by a quantity which is bounded away from zero on $D_r$) . Setting
\begin{equation}\label{eps1}
\varepsilon_1(t,h_{1,n},h_{2,n}):=\frac{1}{n h_{2,n}}\sum_{i=1}^{n}L\left(\frac{t-X_i}{h_{2,n}}f^{1/2}(X_i)\right)
f^{-1/2}(X_i)D(X_i;h_{1,n})I(|t-X_{i}|<h_{2,n}B),
\end{equation}
\begin{equation}\label{eps2}
\varepsilon_2(t,h_{1,n},h_{2,n}):=\frac{1}{n h_{2,n}}\sum_{i=1}^{n}L\left(\frac{t-X_i}{h_{2,n}}f^{1/2}(X_i)\right)
f^{-1/2}(X_i)b(X_i;h_{1,n})I(|t-X_{i}|<h_{2,n}B),
\end{equation}
and
\begin{equation}\label{eps3}
\varepsilon_3(t,h_{1,n},h_{2,n}):=\delta_3(t)+\frac{1}{n h_{2,n}}\sum_{i=1}^{n}L\left(\frac{t-X_i}{h_{2,n}}f^{1/2}(X_i)\right)
f^{1/2}(X_i)\delta_4(X_i;h_{1,n})\textbf{1}\{|t-X_{i}|<h_{2,n}B\},
\end{equation}
we obtain (from (\ref{e1})-(\ref{e4})),
\begin{equation}\label{expans}
\hat f(t;h_{1,n},  h_{2,n})=\bar f(t;h_{2,n})+\frac{1}{2}\varepsilon_1(t)+\frac{1}{2}\varepsilon_2(t)+\varepsilon_3(t).
\end{equation}
By the comments above, the $\varepsilon_2$ and $\varepsilon_3$ terms will be of smaller order than $\varepsilon_1$. $\varepsilon_1$ itself has a $U$-process structure, and the linear term in its Hoeffding decomposition will be the dominant term. This is the content of the lemmas that follow.

\begin{lemma}\label{lemma1}
For $i=2,3$,
$$\sup_{t\in D_r^\varepsilon}|\varepsilon_i(t,h_{1,n},h_{2,n})|=O_{\rm a.s.}(n^{-4/9}) \ \  uniformly\ in\ \ f\in{\cal P}_C$$
for all $C<\infty$.
\end{lemma}

\begin{proof}
We begin with $i=2$. Because the function $L$ is of bounded variation and $b(t;h_{1,n})$ satisfies inequality (\ref{classic2}), it follows (see the Appendix) that the classes of functions
\begin{equation}\label{qu}
{\cal Q}_n:=\left\{Q(x)=L\left(\frac{t-x}{h_{2,n}}f^{1/2}(x)\right)
f^{-1/2}(x)b(x;h_{1,n})I(|t-x|<h_{2,n}B):t\in D_r\right\}
\end{equation}
are of VC type with the same characteristics $A$ and $v$, for envelopes of the order of  $M(K,r)\|f''\|_\infty h_{1,n}^2$, where $M$ depends on $r$ and $K$ only (in particular,  through $L$). If we set
 $$Q_i(t)=L\left(\frac{t-X_i}{h_{2,n}}f^{1/2}(X_i)\right)
f^{-1/2}(X_i)b(X_i;h_{1,n})I(|t-X_i|<h_{2,n}B)$$
it then follows (by the bound (\ref{classic2}) on $b$, boundedness of $L$ and boundedness away from zero of $f$ on $D_r$), that
$$\sup_{t\in D_r}E|Q_i(t)|\lessim \|f''\|_\infty h_{1,n}^2h_{2,n}=\|f''\|_\infty n^{-5/9}(\log n)^{1/9},$$
$$\sup_{t\in D_r}EQ_i^2(t)\lessim \|f''\|^2_\infty h_{1,n}^4h_{2,n}\le  \|f''\|^2_\infty n^{-1}(\log n)^{1/9},\ \ \sup_{t\in D_r}|Q_i(t)|\lessim \|f''\|_\infty h_{1,n}^2=\|f''\|_\infty n^{-4/9},$$
where in these bounds we ignore multiplicative constants that do not depend on $f$. We have
\begin{eqnarray*}
\sup_{t\in D_r}\left|\epsilon_2(t;h_{1,n},  h_{2,n})\right|
&\le &\sup_{t\in D_r}\left|\frac{1}{nh_{2,n}}\sum_{i=1}^{n} [Q_i(t)-EQ_i(t)]\right|+ \sup_{t\in D_r}\frac{1}{h_{2,n}}|EQ_1(t)|\\
&\lessim &  \sup_{t\in D_r}\left|\frac{1}{nh_{2,n}}\sum_{i=1}^{n} [Q_i(t)-EQ_i(t)] \right|+\|f''\|_\infty n^{-4/9},
\end{eqnarray*}
and Talagrand's inequality (\ref{tal}) gives that for $0<\delta\le 4/9$,
$$\sum_n\sup_{f\in{\cal P}_C}{\Pr}_f\left\{\sup_{t\in D_r}\left|\sum_{i=1}^{n} [Q_i(t)-EQ_i(t)] \right|\ge n^{\delta}\right\}\le
C_2\sum_n\exp\left(-\frac{C_3n^{2\delta}}{C^2(\log n)^{1/9}}\right)<\infty.$$
Since $n^{\delta}<nh_{2,n}n^{-4/9}$, we conclude
$$\sup_{t\in D_r}\left|\epsilon_2(t;h_{1,n},  h_{2,n})\right|=O_{\rm a.s.}(n^{-4/9})\ \ {\rm uniformly\ in}\ \ f\in{\cal P}_C$$
proving the lemma for $\varepsilon_2$. Note that $h_{1,n}^2\simeq n^{-4/9}$ plays a critical role in this estimation.

Next, from (\ref{eps3}) we see that $\varepsilon_3$ consists of four sums, the three that define $\delta_3$ and one involving $\delta_4$ (multiplied by bounded terms and by the indicator of $|X_i-t|\le h_{2,n}B$). The  three terms from $\delta_3$ involve, instead of $\delta_4$, respectively $\delta_2$, $\delta^2$ and $\delta_2\delta$ (see (\ref{e1})-(\ref{e3})).  We have from (\ref{delta}), (\ref{classic1}) and (\ref{classic2}) that
$$\sup_{t\in D_r^\varepsilon}\delta^2_n=O_{\rm a.s.}\left(n^{-7/9}\log n\right)\ \ {\rm uniformly\ in}\ \ f\in{\cal P}_C,$$
that the same is true for $\delta_4$ by (\ref{delta4}), and, moreover, by (\ref{delta2}),
$$|\delta_2(t,X_i)|I(|t-X_i|\le h_{2,n}B)\le \frac{\|K''\|_\infty}{2}B^2\|f\|_\infty\delta^2(X_i)=O_{\rm a.s.}\left(n^{-7/9}\log n\right)\ \ {\rm uniformly\ in}\ \ f\in{\cal P}_C.$$
then, if we define $\tilde Q_i(t)$ by
$$\varepsilon_3(t,h_{1,n},h_{2,n})=\frac{1}{nh_{2,n}}\sum_{i=1}^n\tilde Q_i(t),$$
we have
$$ \sup_{t\in D_r}|\tilde Q_i(t)|=O_{\rm a.s.}\left(n^{-7/9}\log n\right)\ \ {\rm uniformly\ in}\ \ f\in{\cal P}_C,$$
and therefore,
$$\sup_{t\in D_r}\left|\epsilon_3(t;h_{1,n},  h_{2,n})\right|=O_{\rm a.s.}( h_{2,n}^{-1}n^{-7/9}\log n)\ \ {\rm uniformly\ in}\ \ f\in{\cal P}_C,
$$
proving the lemma for $\varepsilon_3$ as $h_{2,n}^{-1}n^{-7/9}\log n<< n^{-4/9}$.
 \end{proof}

\begin{lemma}\label{eps1-t}
Let
\begin{eqnarray}\label{tien}
&&T(t; h_{1,n}, h_{2,n})=\\
&&\frac{1}{nh_{1,n}h_{2,n}}\sum_{i=1}^{n}
E_{X}\Big[f^{-1/2}(X)\Big\{K\Big(\frac{X-X_i}{h_{1,n}}\Big)-E_YK\Big(\frac{X-Y}{h_{1,n}}\Big)\Big\}
 L\Big(\frac{t-X}{h_{2,n}}f^{1/2}(X)\Big)I(|t-X|\le h_{2,n}B)\Big],\notag
\end{eqnarray}
where $L(z)=K(z)+zK^\prime(z)$.
Then,   $$\sup_{t\in D_r}\left|\varepsilon_1(t,h_{1,n},h_{2,n})-T(t;h_{1,n},h_{2,n})\right|=o_{\rm a.s.}(n^{-4/9})\ \ {\rm uniformly\ in}\ \ f\in{\cal P}_C.$$
\end{lemma}

\begin{proof}
Given a function $H$ of two variables, and two i.i.d. random variables $X$ and $Y$ such that $H(X,Y)$ is integrable, we recall that the second order Hoeffding projection of $H(X,Y)$ is $$\pi_2(H)(X,Y)=H(X,Y)-E_XH(X,Y)-E_YH(X,Y)+EH.$$
We also recall the $U$-statistic notation
$$U_n(H)=\frac{1}{n(n-1)}\sum_{1\le i\ne j\le n}H(X_i,X_j),$$
where the variables $X_i$ are i.i.d.
Set
$$H_t(X,Y):=L\left(\frac{t-X}{h_{2,n}}f^{1/2}(X)\right)f^{-1/2}(X)K\left(\frac{X-Y}{h_{1,n}}\right)I(|t-X|\le h_{2,n}B).$$
Then,
\begin{equation}\label{hoef1}
\frac{n^2h_{1,n}h_{2,n}}{n(n-1)}\varepsilon_1(t,h_{1,n},h_{2,n})=\frac{1}{n(n-1)}\sum_{i=1}^n (H_t(X_i,X_i)-E_YH_t(X_i,Y))+U_n(H_t-E_YH_t(\cdot,Y))
\end{equation}
(decomposition of a $V$-statistic into the diagonal term and a $U$-statistic).
Now, notice that
\begin{eqnarray}\label{hoef2}
U_n(H_t-E_YH_t(X_i,Y))&=&U_n\left(\pi_2(H_t(\cdot,\cdot))+(E_XH_t(X,\cdot)-EH)\right)\notag\\
&=&U_n(\pi_2(H_t(\cdot,\cdot))+h_{1,n}h_{2,n}T(t;h_{1,n},h_{2,n})
\end{eqnarray}
So, we now must handle the diagonal term, a completely centered or canonical $U$-process and (in the next lemma) the empirical process $T$.

\noindent{\it Diagonal term.} Note that if we define $\bar Q_i$ such that
$$\frac{1}{n^2h_{1,n}h_{2,n}}\sum_{i=1}^n (H_t(X_i,X_i)-E_YH_t(X_i,Y)):=\frac{1}{n^2h_{1,n}h_{2,n}}\sum_{i=1}^n\bar Q_i(t),$$
then we have
$$\sup_{t\in D_r}|E\bar Q_1(t)|\lessim h_{2,n},\ \ \sup_{t\in D_r}E\bar Q_1^2(t)\lessim h_{2,n}, \ \ \sup_{t\in D_r}|\bar Q_1(t)|\lessim 1,$$
where as usual we overlook multiplicative constants that do not depend on $f$, and the last bound does not depend on $n$. So,
$$\sup_{t\in D_r}\frac{1}{n^2h_{1,n}h_{2,n}}\left|\sum_{i=1}^n\bar Q_i(t)\right|\lessim\frac{1}{n^2h_{1,n}h_{2,n}}\sup_{t\in D_r}\left|\sum_{i=1}^n(\bar Q_i(t)-E\bar Q_1(t))\right|+\frac{1}{nh_{1,n}}.$$
The supremum part correspond to the empirical process over the class of functions of $x$
\begin{equation}\label{qubar}
\bar{\cal Q}_n=\left\{L\left(\frac{t-x}{h_{2,n}}f^{1/2}(x)\right)f^{-1/2}(x)\left(K(0)-EK\left(\frac{x-X}{h_{1,n}}\right)\right)
I(|t-x|\le h_{2,n}B): t\in D_r\right\}.
\end {equation}
These classes are VC type with the same characteristics $A$ and $v$ that do not depend on $f$, and with the same envelope, that depends only on $K$ and $r$ (see the Appendix).
Then,  Talagrand's inequality gives, as in previous instances, that, for some $\delta>0$,
$$\sum_n\sup_f{\Pr}_f\left\{\sup_{t\in D_r}\left|\sum_{i=1}^n(\bar Q_i(t)-E\bar Q_1(t))\right|>n^{(4+\delta)/9}\right\}\le
C_2\sum_n\exp\left(-C_3\frac{n^{(8+2\delta)/9}}{nh_{2,n}}\right)<\infty,$$
which, since $n^2h_{1,n}h_{2,n}n^{-4/9}>n^{(4+\delta)/9}$ and since $nh_{1,n}>> n^{4/9}$,
yields
\begin{equation}\label{diag}
\sup_{t\in D_r}\frac{1}{n^2h_{1,n}h_{2,n}}\left|\sum_{i=1}^n (H_t(X_i,X_i)-E_YH_t(X_i,Y))\right|=o_{a.s.}(n^{-4/9})\ \ {\rm uniformly\ in}\ \ f\in{\cal P}_C.
\end{equation}

\medskip

\noindent{\it The canonical $U$-statistic term.} We will use Major's exponential bound (\ref{major}) for canonical $U$-processes over VC type classes of functions. In our case, since the class of functions  $\{H_t:t\in D_r\}$ is uniformly bounded and of VC type (see the Appendix) we can apply Major's exponential bound to $\sup_{t\in D_r}|U_n(\pi_2(H_t))|$. Since, as is easy to check,
$$EH_t^2(X,Y)\le 2B\|L\|_\infty^2\|f\|_\infty\|K\|_2^2 h_{1,n}h_{2,n},$$
we can take, for $C$ such that $\|f\|_\infty\le C$, $\sigma^2\simeq C h_{1,n}h_{2,n}$ and $t=C^{1/2}n^{1+\delta}\sqrt{h_{1,n}h_{2,n}}$ for a small $\delta>0$, to have, from (\ref{major}),
$$\sum_n\sup_{f:\|f\|_\infty\le C}{\Pr}_f\left\{\sup_{t\in D_r}|U_n(\pi_2(H_t))|>Cn^{\delta-1}\sqrt{h_{1,n}h_{2,n}}\right\}\le
C_2\sum_n\exp\left(-C_3n^\delta\right)<\infty.$$
Since $\frac{n^{\delta -1}}{\sqrt{h_{1,n}h_{2,n}}}<<n^{-4/9}$ (we can take $\delta$ so that this is true), we obtain
\begin{equation}\label{ust}
\sup_{t\in D_r}\frac{1}{h_{1,n}h_{2,n}}|U_n(\pi_2(H_t))|=o_{\rm a.s.}(n^{-4/9})\ \ {\rm uniformly\ in}\ \ f\in{\cal P}_C.
\end{equation}

\end{proof}

\medskip

The following lemma will conclude the analysis of (\ref{diff}).
\begin{lemma}\label{lemma3}
With $T$ as defined in Lemma \ref{eps1-t}, we have
$$\sup_{t\in D_r}|T(t;h_{1,n},h_{2,n})|=O_{\rm a.s.}\left(\left(\frac{\log n}{n}\right)^{4/9}\right)\ \ {\rm uniformly\ in}\ \ f\in{\cal P}_C.$$
\end{lemma}

\begin{proof}
Note that
$$T(t;h_{1,n},h_{2,n})=\frac{1}{nh_{1,n}h_{2,n}}\sum_{i=1}^{n}(g(t,X_i)-Eg(t,X))$$
where
\begin{equation}\label{ge}
g(t,x)=E_{X}\Big[f^{-1/2}(X)K\Big(\frac{X-x}{h_{1,n}}\Big)
 L\Big(\frac{t-X}{h_{2,n}}f^{1/2}(X)\Big)I(|t-X|\le h_{2,n}B)\Big].
 \end{equation}
By (\ref{entg}) in the Appendix, the class of functions $\{g(t,\cdot):t\in D_r\}$ is of VC type for the envelope $\|L\|_V\|K\|_2h_{1,n}^{1/2}$ and the characteristics $A=R$ and $v=22$, and the lemma will follow by application of Talagrand's inequality.  We just need to estimate $Eg^2(t,X)$. We have, making several natural changes of variables,
 \begin{eqnarray}
 Eg^2(t,X_1)&=&
E\Big\{\int f^{1/2}(x)K\Big(\frac{x-X_1}{h_{1,n}}\Big)L\Big(\frac{t-x}{h_{2,n}}f^{1/2}(x)\Big)
I(|t-x|<h_{2,n}B)dx\nonumber\\
&&~~~~~~~~~~\times \int f^{1/2}(y)K\Big(\frac{y-X_1}{h_{1,n}}\Big)L\Big(\frac{t-y}{h_{2,n}}f^{1/2}(y)
\Big)I(|t-y|<h_{2,n}B)
dy\Big\}\nonumber\\
&=&h_{1,n}^2\int\int\int f^{1/2}(t-h_{1,n}v_1)L\Big(\frac{h_{1,n}}{h_{2,n}}v_1f^{1/2}(t-h_{1,n}v_1)\Big)
f^{1/2}(t-h_{1,n}v_2)\nonumber\\
&&~~~~~~~~~~\times L\Big(\frac{h_{1,n}}{h_{2,n}}v_2f^{1/2}(t-h_{1,n}v_2)\Big)
K\Big(\frac{t-u}{h_{1,n}}-v_1\Big)K
\Big(\frac{t-u}{h_{1,n}}-v_2\Big)\nonumber\\
&&~~~~~~~~~~~~~~~\times I\Big(\frac{h_{1,n}}{h_{2,n}}|v_1|<B\Big)
I\Big(\frac{h_{1,n}}{h_{2,n}}|v_2|<B\Big)f(u)dudv_1dv_2\nonumber\\
&=&h_{1,n}^3\int\int\int f^{1/2}(t-h_{1,n}v_1)L\Big(\frac{h_{1,n}}{h_{2,n}}v_1f^{1/2}(t-h_{1,n}v_1)\Big)
f^{1/2}(t-h_{1,n}v_2)\nonumber\\
&&~~~~~~~~~~\times L\Big(\frac{h_{1,n}}{h_{2,n}}v_2f^{1/2}(t-h_{1,n}v_2)\Big)K(v)K(v+v_1-v_2)\nonumber\\
&&~~~~~~~~~~~~~~~\times I\Big(\frac{h_{1,n}}{h_{2,n}}|v_1|<B\Big)
I\Big(\frac{h_{1,n}}{h_{2,n}}|v_2|<B\Big)f(t-h_{1,n}v_1-h_{1,n}v)dvdv_1dv_2\nonumber\\
&\le& h_{1,n}^3||f||_\infty^2\int\int\int \left|L\Big(\frac{h_{1,n}}{h_{2,n}}v_1f^{1/2}(t-h_{1,n}v_1)\Big)
 L\Big(\frac{h_{1,n}}{h_{2,n}}v_2f^{1/2}(t-h_{1,n}v_2)\Big)\right|\notag\\
 &&~~~~~~~~~~\times K(v)K(v+v_1-v_2)
I\Big(\frac{h_{1,n}}{h_{2,n}}|v_1|<B\Big)
I\Big(\frac{h_{1,n}}{h_{2,n}}|v_2|<B\Big)dvdv_1dv_2\nonumber\\
&= &h_{1,n}^3||f||_\infty^2\int\int\int \left|L\Big(\frac{h_{1,n}}{h_{2,n}}(w+v_2)f^{1/2}(t-h_{1,n}w-h_{1,n}v_2)\Big)
L\Big(\frac{h_{1,n}}{h_{2,n}}v_2f^{1/2}(t-h_{1,n}v_2)\Big)\right|\nonumber\\
&&~~~~~~~~~~\times K(v)K(v+w)I\Big(\frac{h_{1,n}}{h_{2,n}}|w+v_2|<B\Big)
I\Big(\frac{h_{1,n}}{h_{2,n}}|v_2|<B\Big)dvdwdv_2\nonumber\\
&=&h_{1,n}^2h_{2,n}||f||_\infty^2\int\int\int\left|L\Big((\frac{h_{1,n}}{h_{2,n}}w+z) f^{1/2}(t-h_{1,n}w-h_{2,n}z)\Big)\right|\nonumber\\
&&~~~~~~~~~~\times \left|L\Big(z f^{1/2}(t-h_{2,n}z)\Big)\right|K(v)K(v+w)I\Big(|\frac{h_{1,n}}{h_{2,n}}w+z|<B\Big)
I\big(|z|<B\big)dvdwdz\nonumber\\
&\le& 2h_{1,n}^2h_{2,n}||f||_\infty^2 B(||K||_\infty+B\|f\|_\infty^{1/2}||K^\prime||_\infty)^2.
\end{eqnarray}
So we can take $\sigma^2=c_2^2(1\vee\|f\|_\infty^3)h_{1,n}^2h_{2,n}$, where $c_2$ depends only on $K$. Since, as indicated above, the collection of functions $g(t,\cdot)$, $t\in D_r$, is VC for an envelope of the order $h_{1,n}$,
 Talagrand's inequality (\ref{tal}) implies that there exist finite positive constants  $c_0,c_1$ such that, with $C_i$ as in (\ref{tal}), if
$$C_1(1\vee\|f\|_\infty^{3/2})\sqrt{n}h_{1,n}h_{2,n}^{1/2}\sqrt{\log \frac{c_0h_{1,n}^{1/2}}{c_2h_{1,n}h_{2,n}^{1/2}}}<u<C_2\frac{n(1\vee\|f\|_\infty^3)c_2^2h_{1,n}^2h_{2,n}}{c_1h_{1,n}^{1/2}}$$
then
$${\Pr}_f\left\{\left\|\sum_{i=1}^n(g(t,X_i)-Eg(t,X))\right\|_{D_r}\ge u\right\}\le C_2\exp\left(-\frac{C_3u^2}{c_2^2(1\vee\|f\|_\infty^3)nh_{1,n}^2h_{2,n}}\right).$$
The condition on $u$ can be written as
$$ C'_1(1\vee\|f\|_\infty^{3/2}) n^{2/9}(\log n)^{5/9}<u<C_2'(1\vee\|f\|_\infty^3)n^{5/9}(\log n)^{1/9},$$
and if we take $u=M(1\vee\|f\|_\infty^{3/2})n^{2/9}(\log n)^{5/9}$ for some large enough $M$, then
$$\sum \exp\left(-\frac{C_3u^2}{c_2^2(1\vee\|f\|_\infty^3)nh_{1,n}^2h_{2,n}}\right)= \sum e^{-M^2C_3(\log n)/c_2^2}<\infty$$
uniformly in $f$.
Hence,
$$\sum\sup_{f\in{\cal P}_C}{\Pr}_f\left\{\left\|\sum_{i=1}^n(g(t,X_i)-Eg(t,X))\right\|_{D_r}\ge M(1\vee C^{3/2})n^{2/9}(\log n)^{5/9}\right\}$$
$$\le\sum\sum e^{-M^2C_3(\log n)/c_2^2}<\infty$$
This shows that $T(t;h_{1,n}h_{2,n})$ is asymptotically a.s. of the order of $n^{2/9}(\log n)^{5/9}/(nh_{1,n}h_{2,n})=[(\log n)/n]^{4/9}$ uniformly in $f\in{\cal P}_C$.
\end{proof}

\medskip
From (\ref{expans}) and Lemmas \ref{lemma1}, \ref{eps1-t} and \ref{lemma3}, we obtain:

\begin{proposition}\label{real-ideal}
Under the Assumptions \ref{ass3}, for any $C<\infty$ we have:
$$\sup_{t\in D_r}\left|\hat f(t;h_{1,n},h_{2,n})-\bar f(t;h_{2,n})-T(t;h_{1;n},h_{2,n})\right|=o_{\rm a.s.}\left(\left(\frac{\log n}{n}\right)^{4/9}\right)\ \ {\rm uniformly\ in}\ \ f\in{\cal P}_C.$$
and in particular,
$$\sup_{t\in D_r}\left|\hat f(t;h_{1,n},h_{2,n})-\bar f(t;h_{2,n})\right|=O_{\rm a.s.}\left(\left(\frac{\log n}{n}\right)^{4/9}\right)\ \ {\rm uniformly\ in}\ \ f\in{\cal P}_C.$$
\end{proposition}

\begin{remark}
{ \rm a) We should remark that if we undersmooth the preliminary estimator a little more, by taking $h_{1,n}=n^{-(2+\eta)/9}$ with $0<\eta< 2$, then the three lemmas above are true and moreover we have
$\sup_{t\in D_r^\varepsilon}|\varepsilon_i(t,h_{1,n},h_{2,n})|=o_{\rm a.s.}(n^{-4/9})$ in Lemma \ref{lemma1}. So, for such $h_{1,n}$ the order of the first term in Proposition \ref{real-ideal} is actually
$o_{\rm a.s.}\left(n^{-4/9}\right)$. This is at odds with condition (9) in Hall, Hu and Marron (1995), as their condition does not necessarily imply undersmoothing of the preliminary estimator. b) It is worth mentioning that Proposition \ref{real-ideal} does require that the indicators $I(|t-X_i|\le h_{2,n}B)$ be part of the definition of (\ref{ideal0}) and (\ref{realest0}): in fact none of the three lemmas in its proof seem to go through without it. This condition is required as well for the bias of the ideal estimator,  but it is not necessary for its variance part. }
\end{remark}

\medskip

Now we can complete the proof of the main theorems \ref{main0} and \ref{mainu}. Only the stronger Theorem \ref{mainu} requires proof:

\medskip
\noindent {\bf Proof of Theorem \ref{mainu}.} Proposition \ref{real-ideal} and Theorem \ref{unifidealthm} together give (\ref{main1'}). The limit (\ref{main2'}) can be easily derived from (\ref{main1'}), as follows. By (\ref {classic1}) and (\ref{classic2}), the preliminary estimator satisfies
\begin{equation}
\label{prelimunif}
\sup_{t\in D_r}|\hat f(t;h_{1,n})-f(t)|=O_{\rm a.s. }\left(\frac{\sqrt{\log n}}{n^{7/18}} \right)\ \
{\rm uniformity\ in}\   {\cal D}_{C,z}
\end {equation}
 for all $C<\infty$, $z$ and $r$. Now, for all $n$ large enough,
 on the event
$$\left\{\sup_{n\ge k}\frac{n^{7/18}}{\sqrt{\log n}}||\hat f(t;h_{1,n})-f(t)||_\infty\le \lambda_1\right\}$$
we have $\hat D_r^n\subset D_r$ for all $n\ge k$, and therefore,
\begin{eqnarray*}
&&{\Pr}_f\left\{\sup_{n\ge k}\big(\frac{ n}{\log n}\big)^{4/9}\|\hat f(t;h_{1,n},  h_{2,n},\omega)-f(t)\|_{\hat D_r^n}>\lambda_2\right\}\\
&&\le{\Pr}_f\left\{\sup_{n\ge k}\big(\frac{ n}{\log n}\big)^{4/9}\|\hat f(t;h_{1,n},  h_{2,n},\omega)-f(t)\|_{ D_r}>\lambda_2\right\}\\
&&~~~~~~~~~~+{\Pr}_f\left\{\sup_{n\ge k}\frac{n^{7/18}}{\sqrt{\log n}}||\hat f(t;h_{1,n})-f(t)||_\infty> \lambda_1\right\}.
\end{eqnarray*}
Now, there exist $\lambda_1$ and $\lambda_2$ such that the limit of the sup over ${\cal D}_{C,z}$ of the first probabilities is zero by (\ref{main1'}), and  the limit of the sup of the second ones over the same set is also zero by (\ref{prelimunif}), proving (\ref{main2'}). \qquad $\blacksquare$

\section{Appendix: Some Vapnik-\v Cervonenkis classes of functions and their exponential bounds}
Let $\cal F$ be a collection of uniformly bounded measurable functions on $(S,{\cal S})$. We say that $\cal F$ is of VC type with respect to an envelope $F$ if there exist constants $A$, $v$ positive such that for all probability measures $Q$ on $\cal S$,
$$N({\cal F}, L_2(Q),\varepsilon)\le\left(\frac{A\|F\|_{L_2(Q)}}{\varepsilon}\right)^v,\ \  0<\varepsilon<1,$$
where $F\ge |f|$ for all $f\in \cal F$ and $N({\cal F},L_2(Q),\varepsilon)$ denotes the smallest number of $L_2(Q)$-balls of radius at most $\varepsilon$ required to cover $\cal F$. (See e.g., de la Pe\`na and Gin\'e (1999).) It turns out that empirical processes or $U$-processes indexed by these classes of functions are very well behaved, particularly if $F$ is uniformly bounded and if the class $\cal F$ is countable. For instance, we have the following version of an exponential inequality of Talagrand (1996)  from Einmahl and Mason (2000) and Gin\'e and Guillou (2001, 2002). Let $P$ be a probability measure on $S$ and let $X_i:S^{\mathbb N}\mapsto S$ be the coordinate functions of $S^{\mathbb N}$, which are i.i.d. (P), and set $\Pr=P^{\mathbb N}$. If
 the class $\cal F$ is VC type, bounded and  countable, then there exist $0<C_i<\infty$, $1\le i\le 3$, depending on $v$ and $A$ such that, for all $t$ satisfying
 $$C_1\sqrt{n}\sigma\sqrt{\log\frac{2\|F\|_\infty}{\sigma}}\le t\le \frac{n\sigma^2}{\|F\|_\infty},$$
 we have
 \begin{equation}\label{tal}
\Pr\left\{\max_{1\le k\le n}\left\|\sum_{i=1}^k(f(X_i)-Pf)\right\|_{\cal F}>t\right\}\le C_2\exp\left(-C_3\frac{t^2}{n\sigma^2}\right),
\end{equation}
where
$$\|F\|_\infty\ge \sigma^2\ge\|{\rm Var}_P(f)\|_{\cal F}.$$
(Talagrand (1996) states his inequality only for the sum over $n$, but the same works for the maximum of the partial sums up to $n$ by a (sub)martingale argument that can be carried out because these inequalities are obtained  by integrating bounds on the moment generating function -see e.g., Ledoux (2001).) Major (2006) also has a similar inequality for classes of functions of several variables. We will state his inequality for bounded VC type classes of functions of two variables only. Let $\cal F$ be such a class of functions and let $\|F\|_\infty^2\ge \sigma^2\ge \|{\rm Var}(f(X_1,X_2))\|_{\cal F}$. Let $\pi_2^P(f)(x,y)=f(x,y)-Ef(X,y)-Ef(x,X)+Ef(X,Y)$. Then, if $\cal F$ is a uniformly bounded, countable class of VC type, there exist  $0<C_i<\infty$, $1\le i\le 3$, depending on $v$ and $A$ such that, for all $t$ satisfying
$$C_1n\sigma\log\frac{2\|F\|_\infty}{\sigma}\le t\le \frac{n^2\sigma^3}{\|F\|_\infty^2}$$
we have
\begin{equation}\label{major}
\Pr\left\{\left\|\sum\sum_{1\le i\ne j\le n}\pi_2^Pf(X_i,X_j)\right\|_{\cal F}>t\right\}\le C_2\exp\left(-C_3\frac{t}{n\sigma}\right).
\end{equation}
Major states the theorem for $\{\pi_2^Pf\}$ of VC type, but it is easy to see that if $\cal F$ is VC type for $F$ then $\{\pi_2^Pf:f\in {\cal F}\}$ is VC type for the envelope $4F$.

It is also worth mentioning that (much easier to prove) moment bounds for the above quantities are also available (e.g. in Gin\'e and Mason (2007) and references therein) and that they can be used instead of Talagrand and Major's inequalities if one is only interested in the `in probability' version of Theorems \ref{main0} and \ref{mainu}.

We now show that the classes of functions appearing in the previous sections are of VC type, and the suprema countable. We will do this in all detail for the class $\cal F$ in (\ref{entr0}), and will give indications for the rest of the classes of functions used.

First we observe that the sup inside the probability bound in (\ref{eq1}) is actually a  supremum over the set $\{t\in\mathbb Q, h\in \mathbb Q\cap[h_{2^k},h_{2^{k-1}})\}$ by the continuity properties of $K$ and the indicator of $|t-X_i|<h B$. This observation applies to all the other classes of functions in the previous two sections.

\begin{lemma}\label{vc1}
Let $\cal F$ be as in (\ref{entr0}) with $K$ and $f$ satisfying Assumptions \ref{ass1}. Then, there exists a universal constant $R$ such that for every Borel probability measure $Q$ on $\mathbb R$,
\begin{equation}\label{ivc1}
N({\cal F},L_2(Q),\varepsilon)\le\left(\frac{R\|K\|_V\|f\|_\infty^{1/2}}{\varepsilon}\right)^{22}
\end{equation}
where $\|K\|_V$ is the total variation norm of $K$, that is, $\cal F$ is of VC type with envelope $\|K\|_V\|f\|_\infty^{1/2}$ with $A=R$ independent of $K$ and $f$ and $v=22$.
\end{lemma}

\begin{proof}
By adding an arbitrarily small strictly increasing function to the positive and negative variation functions of $K$, we have $K=K_1-K_2$ with $K_i$ strictly increasing, positive and bounded, with $\|K_1\|_\infty$ ($\|K_2\|_\infty$) arbitrarily close to the positive (negative) variation of $K$. Let ${\cal K}_1$ be the class of functions obtained from $\cal F$ by replacing $K$ by $K_1$ and deleting the indicator in each of the functions in the class. Then, if we assume $f(x)>0$ for all $x$, the subgraphs of the functions in the class ${\cal K}_1$ have the form
$$
\left\{(x,u):K_{1}\left(\frac{t-x}{h}f^{1/2}(x)\right)f^{1/2}(x)
\ge u \right\}
=\left\{(x,u):\frac{t-x}{h}f^{1/2}(x)\ge K_{1}^{-1}(u/f^{1/2}(x))\right\}
$$
$$
=\left\{(x,u):\frac{tf^{1/2}(x)}{h}-\frac{xf^{1/2}(x)}{h}-K_{1}^{-1}(u/f^{1/2}(x))
\ge0\right\},
$$
and so they are the positivity sets of  functions  from the linear space of functions of the two variables $u$ and $x$ spanned by $f^{1/2}(x)$, $xf^{1/2}(x)$ and $K_1^{-1}(u/f^{1/2}(x))$. Hence, by a result of Dudley (e.g. Proposition 5.1.12 in de la Pe\~na and Gin\'e (1999)) the subgraphs of ${\cal K}_1$ are VC of index 4. If the set $\{x:f(x)= 0\}$ is not empty, the same argument above shows that the class of subsets of $S=\{x:f(x)>0\}\times\mathbb R$, $\left\{(x,u)\in S:K_{1}\left(\frac{t-x}{h}f^{1/2}(x)\right)f^{1/2}(x)
\ge u \right\}$ is VC of index 4, and therefore so is the class of subgraphs of ${\cal K}|_1$, which is obtained from this one by taking the union of each of these sets with the set $\{x:f(x)=0\}\times\{u\le 0\}$ (which is disjoint with all of them).
Therefore, in either case, by the Dudley-Pollard entropy theorem for VC-subgraph classes (e.g., loc. cit. Theorem 5.1.5), we have
$$
N({\cal K}_1, L_2(P), \varepsilon)\le \left(\frac{A\|K_1\|_\infty\|f\|_\infty^{1/2}}{\varepsilon}\right)^8,\ \ 0<\varepsilon\le \|K_1\|_\infty\|f\|_\infty^{1/2}$$
where $A$ is a universal constant, hence,
\begin{equation}\label{vc2}
N({\cal K}_1, L_2(P), \varepsilon)\le \left(\frac{A\|K\|_+\|f\|_\infty^{1/2}}{\varepsilon}\right)^8,\ \ 0<\varepsilon\le \|K\|_+\|f\|_\infty^{1/2}
\end{equation}
where $\|K\|_+$ is the positive variation seminorm of $K$. The analogous bound holds for ${\cal K}_2$, defined with  $K_2$ replacing $K$ in $\cal F$. Since, as is well known, the set ${\cal J}$ of all  indicator functions of intervals in $\mathbb R$ is $VC$ of order 3, we also have
\begin{equation}\label{vc3}
N({\cal J}, L_2(P),\varepsilon)\le\left(\frac{\bar A }{\varepsilon}\right)^6, \ \ 0<\varepsilon\le 1.
\end{equation}
for another universal constant $\bar A$. Now, any $H\in\cal F$ can be written as $H=k_1g-k_2g$ for $k_i\in{\cal K}_i$ and $g\in{\cal J}$, so that, for any probability measure $Q$ we have
\begin{eqnarray*}Q(H-\bar H)^2&=&Q((k_1-k_2)g-(\bar k_1-\bar k_2)\bar g)^2\\
&\le&4Q(k_1-\bar k_1)^2+4Q(k_2-\bar k_2)^2+2\|K\|_V^2\|f\|_\infty Q(g-\bar g)^2.
\end{eqnarray*}
Given $\varepsilon>0$ let $\delta_1=\varepsilon/4$ and $\delta_2=\varepsilon/(2\|K\|_V\|f\|^{1/2})$.
Then, if the collections of functions $k_1^{(1)},\dots, k_{N_1}^{(1)}$ and $k_1^{(2)},\dots, k_{N_2}^{(2)}$ are $L_2(Q)$ $\delta_1$-dense respectively in the classes ${\cal K}_1$, ${\cal K}_2$, and $g_1,\dots,g_{N_3}$ are $L_1(Q)$ $\delta_2$-dense  in the class $\cal J$, with optimal cardinalities $N_i=N({\cal K}_i, L_2(Q),\delta_1)$, $i=1,2$, and $N_3=N({\cal J},L_2(Q),\delta_2)$, then, by the previous inequality, the functions $(k_i^{(1)}-k_j^{(2)})g_l$ are $L_2(Q)$ $\varepsilon$-dense in $\cal F$ . Since there are at most $N_1N_2N_3$ such functions (this estimate may not be optimal), the inequality  (\ref{ivc1}) follows.

\end{proof}

\medskip
A similar result holds for the classes ${\cal Q}_n$ defined by (\ref{qu}) in the proof of Lemma \ref{lemma1}, the classes of functions $\bar{\cal Q}_n$ defined by (\ref{qubar})  and the classes $\{H_t(x,y):t\in D_r\}$ in the proof of Lemma \ref{eps1-t}, as all  these classes have the same structure as $\cal F$ in Lemma \ref{vc1}.

The class of functions ${\cal G}:=\{g(t,\cdot):t\in D_r\}$ where $g$ is defined in (\ref{ge}) in the proof of Lemma \ref{lemma3}, requires some extra considerations. Let $Q$ be any probability measure on the line and let $s,t\in D_r$. Then, using H\"older, we have
$$
E_Q(g(t,x)-g(s,x))^2\le \int E_X\left(f(X)^{-1}K^2\left(\frac{X-x}{h_{1,n}}\right)\right)\times$$
$$\times E_X\left(L\Big(\frac{t-X}{h_{2,n}}f^{1/2}(X)\Big)
I(|t-X|<h_{2,n}B)-L\Big(\frac{s-X}{h_{2,n}}f^{1/2}(X)\Big)
I(|s-X|<h_{2,n}B)\right)^2dQ(x)$$
$$= h_{1,n}\|K\|_2^2\int \left(L\Big(\frac{t-y}{h_{2,n}}f^{1/2}(y)\Big)
I(|t-y|<h_{2,n}B)-L\Big(\frac{s-y}{h_{2,n}}f^{1/2}(y)\Big)
I(|s-y|<h_{2,n}B)\right)^2f(y)dy$$
\begin{equation}=h_{1,n}\|K\|_2^2E_f(\ell_t-\ell_s)^2
\end{equation}
where $\ell_{s}$ and $\ell_{t}$ are functions from the class
$${\cal L}:=\left\{L\left(\frac{t- \cdot}{h}f^{1/2}(\cdot)\right)I(
|t-\cdot|<hB):t\in\mathbb{R}, h>0 \right\}$$
which is VC with a constant envelope by Lemma \ref{vc1}. This lemma then proves that for all $Q$,
\begin{equation}\label{entg}
N({\cal G}, L_2(Q),\varepsilon)\le\left(\frac{R\|L\|_V\|K\|_2h_{1,n}^{1/2}}{\varepsilon}\right)^{22},\ \ 0<\varepsilon<\|L\|_V\|K\|_2h_{1,n}^{1/2},
\end{equation}
in particular, $\cal G$ is VC for the constant envelope $\|L\|_V\|K\|_2h_{1,n}^{1/2}$, with characteristics $A=R$ and $v=22$.

\bigskip
\noindent {\bf Acknowledgement.} We thank Richard Nickl for several useful conversations on the subject of this article.

\medskip

\noindent E. Gin\'e\hfill\break
\noindent Department of Mathematics, U-3009\hfill\break
\noindent University of Connecticut\hfill\break
\noindent Storrs, CT 06269\hfill\break
\noindent gine@math.uconn.edu

\medskip
\noindent H. Sang\hfill\break
\noindent Department of Mathematics\hfill\break
\noindent University of Cincinnati\hfill\break
\noindent Cincinnati, OH 45221 \hfill\break
\noindent sanghn@ucmail.uc.edu
\end{document}